\documentclass[12pt,a4paper]{amsart}
\usepackage[margin=2.5cm,marginpar=2cm]{geometry}

\usepackage[bookmarksopen,bookmarksdepth=3,hidelinks]{hyperref}
\usepackage[nameinlink]{cleveref}

\usepackage{amscd}
\usepackage{amssymb}
\usepackage{mathrsfs}
\usepackage{mathbbol,mathabx}
\usepackage{amsthm}
\usepackage{graphicx}
\usepackage{braket}

\usepackage[alphabetic]{amsrefs}
\usepackage[all,cmtip]{xy}
\usepackage{rotating}
\usepackage{leftidx}

\usepackage[rgb,table,dvipsnames]{xcolor}

\usepackage{tikz-cd}

\usepackage[normalem]{ulem}

\usepackage{enumitem}

\crefname{equation}{}{}
\Crefname{equation}{Equation}{Equations}
\Crefname{lem}{Lemma}{Lemma}

\long\def\delete#1{}

\newcommand{\trivial}[2][]{\if\relax\detokenize{#1}\relax
{\color{red} \vspace{0em} $[$  #2 $]$}
\else
\ifx#1h
\ifcsname showtrivial\endcsname
{\color{orange} \vspace{0em}  $[$ #2 $]$}
\fi
\else {\red Wrong argument!} \fi
\fi\ignorespaces
}

\newcommand{\byhide}[2][]{\if\relax\detokenize{#1}\relax
{\color{orange} \vspace{0em} Plan to delete:  #2}
\else
\ifx#1h\relax\fi
\fi
}

\newcommand\clrr{\color{red}}

\newcommand\lm{\lambda}

\newcommand{\bfone}{\mathbf{1}}

\def\GL{\mathrm{GL}}
\def\SO{\mathrm{SO}}
\def\Sp{\mathrm{Sp}}
\def\rO{\mathrm{O}}
\def\Ind{\mathrm{Ind}}
\def\Nil{\mathrm{Nil}}

\def\tr{\mathrm{tr}}

\def\bC{\mathbb{C}}
\def\bN{\mathbb{N}}

\def\cC{\mathcal{C}}
\def\cO{\mathcal{O}}
\def\ckcO{{\check{\cO}}}
\def\sY{\mathscr{Y}}
\def\bfdd{\mathbf{d}}
\def\sfS{\mathsf{S}}
\def\sfW{\mathsf{W}}
\def\rcc{\mathrm{c}}
\def\fhh{\mathfrak{h}}
\def\fgg{\mathfrak{g}}
\def\foo{\mathfrak{o}}

\def\fsp{\mathfrak{sp}}

\def\fso{\mathfrak{so}}
\newcommand{\BC}{{\mathbb {C}}}

\newcommand{\BN}{{\mathbb {N}}}

\newcommand{\BZ}{{\mathbb {Z}}}

\newcommand{\CO}{{\mathcal {O}}}

\renewcommand{\vsp}{{\vspace{0.2in}}}

\newcommand{\sgn}{\operatorname{sgn}}

\newcommand{\oO}{\operatorname{O}}
\newcommand{\oS}{\operatorname{S}}

\newcommand{\oZ}{\mathcal{Z}}

\newcommand{\oU}{\mathcal{U}}

\newcommand{\g}{\mathfrak g}

\newcommand{\h}{\mathfrak h}
\newcommand{\p}{\mathfrak p}

\renewcommand{\l}{\mathfrak l}
\renewcommand{\t}{\mathfrak t}
\newcommand{\s}{\mathfrak s}

\renewcommand{\o}{\mathfrak o}

\newcommand{\gl}{\mathfrak g \mathfrak l}

\newcommand{\Z}{\mathbb{Z}}
\def\C{\mathbb{C}}
\newcommand{\R}{\mathbb R}

\newcommand{\la}{\langle}
\newcommand{\ra}{\rangle}

\newcommand{\be}{\begin {equation}}
\newcommand{\ee}{\end {equation}}

\newtheorem*{thm*}{Theorem}
\newtheorem{thm}{Theorem}[section]

\newtheorem{lem}[thm]{Lemma}

\newtheorem*{lem*}{Lemma}

\newtheorem{prop}[thm]{Proposition}

\newtheorem{cor}[thm]{Corollary}

\newtheorem*{claim*}{Claim}
\newtheorem{defn}[thm]{Definition}
\newtheorem{dfnl}[thm]{Definition}

\theoremstyle{remark}
\newtheorem*{remark}{Remark}

\newtheorem*{eg*}{Example}

\def\half{{\tfrac{1}{2}}}

\def\orb{\mathrm{Orbit}}

\def\ckfgg{\check{\fgg}}

\def\chico{\chi_{\ckcO}}

\def\tdBV{\tilde{\mathrm{d}}_{\mathrm{BV}}}

\def\tdSP{\tilde{\mathrm{d}}_{\mathrm{SP}}}

\def\msim{\stackrel{m}{\sim}}
\def\osim{\stackrel{o}{\sim}}

\def\Irr{\mathrm{Irr}}

\def\bNilsp{\overline{\Nil}^{\mathrm{sp}}}

\def\Irrms{\Irr^{\mathrm{ms}}}
\def\Irrsp{\Irr^{\mathrm{sp}}}

\def\tdSP{\tilde{\mathrm d}_{\mathrm{SP}}}

\def\Springer{\fO}

\def\Springer{{\mathrm{Springer}}}

\def\iotacell{\iota_{\mathrm{cell}}}

\def\MA{$\mathsf{(MA)}$}
\def\MB{$\mathsf{(MB)}$}

\newtheorem{introtheorem}{\bf{Theorem}}

\begin{document}

\title[]{On the notion of metaplectic Barbasch-Vogan duality}

\author [D. Barbasch] {Dan Barbasch}
\address{Department of Mathematics\\
Cornell University\\
Ithaca, NY14853, USA}
\email{barbasch@math.cornell.edu}

\author [J.-J. Ma] {Jia-Jun Ma}
\address{School of Mathematical Sciences\\
  Xiamen University
  Xiamen, 361005, China}
  \address{Department of Mathematics, Xiamen University Malaysia campus, Sepang, Selangor Darul Ehsan, 43900,  Malaysia}
 \email{hoxide@xmu.edu.cn}

\author [B. Sun] {Binyong Sun}
\address{Institute for Advanced Study in Mathematics, Zhejiang University\\
  Hangzhou, 310058, China}\email{sunbinyong@zju.edu.cn}

\author [C.-B. Zhu] {Chen-Bo Zhu}
\address{Department of Mathematics\\
  National University of Singapore\\
  10 Lower Kent Ridge Road, Singapore 119076} \email{matzhucb@nus.edu.sg}


\begin{abstract}
In analogy with the Barbasch-Vogan duality for real reductive linear
groups, we introduce a duality notion useful for the representation theory of the real metaplectic groups. This is a map on the set of nilpotent orbits in a complex symplectic Lie algebra, whose range consists of the so-called metaplectic special nilpotent
orbits. We relate this duality notion with the theory of primitive
ideals and extend the notion of special unipotent representations to
the real metaplectic groups. We also interpret the duality map in
terms of double cells of Weyl group representations.
  \end{abstract}

\maketitle

\tableofcontents

\section{Introduction and the main results}\label{sec:intro}

An important problem in representation theory is to extend the local Langlands program on the representation theory of linear algebraic groups defined over a local field to nonlinear covers. A key example is of course the metaplectic group and its most famous representations are the oscillator (or Weil) representations \cite{Sha, Weil}. This has been a subject of intense research in recent years, see for instance the foundation paper of Weissman on L-groups and parameters for covering groups \cite{Weis}.  We refer the reader to \cite{GGW} for a discussion on the historical developments in the study of nonlinear covering groups.

The subject of this paper concerns one small but useful aspect of the (extended) Langlands duality in the case of the metaplectic group and it concerns nilpotent orbits. The local field will be archimedean and the duality will be an analogue of the Barbasch-Vogan duality introduced in \cite{BVUni}. A key impetus as well as the main purpose of this paper are to make explicit and emphasize this duality notion, which we view to be basic for the representation theory of the real metaplectic group.

\medskip

We recall the notions relevant to adjoint nilpotent orbits for complex
classical Lie groups (the general linear groups, the orthogonal and the symplectic groups). A standard
    reference is \cite{CM}. We will generally identify an element  of
  the Lie algebra $\g$ of a complex classical Lie group $G$ with the corresponding matrix
  in the standard representation.  An $X\in \mathfrak g$ (or its $G$-orbit) is called
  nilpotent, if the corresponding matrix is nilpotent. The set of nilpotent orbits
  is denoted by $\overline{\mathrm{Nil}}
  (\g\l_n(\BC))$, $\overline{\mathrm{Nil}}(\o_{2n+1}(\BC))$,
  $\overline{\mathrm{Nil}}(\s\p_{2n}(\BC))$ or
  $\overline{\mathrm{Nil}}(\o_{2n}(\BC))$ ($n\in \BN:=\{0,1,2,3,\dots\}$),  and called   of type A, B,
  C, or D, respectively. The orbits are called classical
  nilpotent orbits. Recall that a classical nilpotent orbit is also identified with its
  Young diagram. Briefly, by the Jacobson-Morozov theorem, there is an $\s\l_2$-triple representing the nilpotent orbit (i.e, containing an element of the orbit). View the standard representation of $\g$ as the representation of the $\s\l_2$-triple. The dimensions of the irreducible $\s\l_2$-subrepresentations then give the rows of the Young diagram.
  The Young diagram is (again) called type A, B, C, D if the corresponding nilpotent orbit is of this type. Specifically even sized rows occur an even
  number of times for types B and D, and odd sized rows occur an even number of times for type C.

Nilpotent orbits under the orthogonal group may differ from orbits of the
special orthogonal group in case the matrix size $m=4n>0$.
Each nilpotent $\oO_{4n}(\C)$-orbit
$\mathcal O\subset \o_{4n}(\C)$ splits into one or two nilpotent
$\SO_{4n}(\C)-$orbits. The cases when the $\oO_{4n}(\C)-$orbit splits into
two $\SO_{4n}(\C)-$orbits are called very even. For
these orbits, the
corresponding Young diagram consists of even rows only, each row size
occurring an even number of times.

A nilpotent orbit $\CO$ is said to be special if it is special in the
sense of  Lusztig (\cite{L79,L82}). In the case of $\oO_{4n}(\C)$, by convention every very even nilpotent orbit is special, since both $\SO_{4n}(\C)-$orbits in it are special. The set of special orbits is
indicated by a superscript $\mathrm{sp}$.

\medskip
We will use the fact (\cite[Section 6.3]{CM}) that a nilpotent orbit of type B, C, or D is special if and only if the transpose of its Young diagram is of type B, C or C, respectively. They are a
  paraphrase of the notion of Lusztig who uses symbols. Note the the curious situation for type $\mathrm{D}$ in the above characterization (see \cite[page 100]{CM}), which is addressed by the notion about to be defined.

We will consider the complex symplectic group. We let $\g:=\s\p_{2n}(\C)$ ($n\geq 0$).
\begin{defn}\label{d:msp}
A nilpotent orbit of type $\mathrm C$ is said to be metaplectic special if the transpose of its Young diagram is of type $\mathrm D$.
\end{defn}

\begin{remark} The notion of metaplectic special appears earlier in \cite{Mo96} (where it is called anti-special) and in \cite{JLS}.
\end{remark}

As usual, the universal Cartan subalgebra $\t$ of $\g$ is identified
with $\BC^n$, and the set of positive roots is identified with
\begin{equation}\label{eq:roots}
  \Phi^+:=\{e_i\pm e_j\mid 1\leq i<j\leq n\}\sqcup \{2e_i\mid 1\leq i\leq n\}\subset \t^*\quad (\textrm{$*$ indicates the dual space}).
\end{equation}
Here $e_1, e_2,\cdots, e_n$ is the standard basis of $\BC^n$, and we
identify the dual space $(\BC^n)^*$ with $\BC^n$ so that the basis $e_i$ are self-dual. The Weyl group $\sfW_n\subset \GL_n(\BC)$ of $\g$, called the Weyl group of type $\mathrm C_{n}$, is generated by
all the permutation matrices and the diagonal matrices of order
$2$. Write $\oZ(\g)$ for the center of the universal enveloping
algebra $\oU(\g)$ of $\g$.  Through the Harish-Chandra isomorphism, we
have an identification
\begin{equation}\label{zg}
 \oZ(\g)=\left(\oS(\C^n)\right)^{\sfW_n},
\end{equation}
where $\oS$ indicates the symmetric algebra, and a superscript group indicates invariants under the group action.
A (algebraic) character of $\oZ(\g)$ is thus represented by an orbit
of $\sfW_n$ in $ \BC^n$.

\begin{defn}\label{def:mpli}
A character $\chi: \oZ(\g)\rightarrow \BC$ is said to be metaplectic integral if it is represented by an element in $(\frac{1}{2}+\BZ)^n$.
\end{defn}

The main results of this note concern the primitive ideals of $\oU(\g)$ and their associated varieties, for
$\g=\s\p_{2n}(\C)$. We refer the reader to \cite{Dix} as a general reference.

Recall that for any complex reductive Lie algebra $\g_1$, the associated variety of a two-sided ideal of $\oU(\g_1)$ is the subvariety in $\g_1^*$ of zeroes of the associated graded ideal in $\textrm{S}(\g_1)$, the symmetric algebra on $\g_1$. Here the grading is defined using the standard filtration of $\oU(\g_1)$. Recall also the basic result of Borho-Brylinski \cite{BB} and Joseph
\cite{J85} that the associated variety of a primitive ideal of
$\oU(\g_1)$ is the closure of a single nilpotent orbit in $\g_1^*$. 

We will identify $\g =\s\p_{2n}(\BC) $ with $\g^*$ via half of the trace form
\[
  \la X, Y\ra:=\half\tr(XY),\qquad X, Y\in \g.
\]
(This half of the trace form induces the ``standard" form on the universal Cartan subalgebra $\t=\mathbb C^n$.)

\begin{introtheorem}\label{thm13}
Let $\g =\s\p_{2n}(\BC)$ and $I$ be a primitive ideal of $\oU(\g)$ with a metaplectic integral infinitesimal character.  Then the associated variety of $I$ is the closure of a metaplectic special nilpotent orbit in $\g^*$.
\end{introtheorem}

\begin{remark} Theorem \ref{thm13} should be viewed as an analogue of the result of Barbasch and Vogan on representations of complex semisimple groups with integral infinitesimal characters (\cite[Definition 1.10]{BVUni} and remarks immediately after).  For analogous results in the case of $p-$adic metaplectic groups, we refer the reader to the work of M{\oe}glin \cite[Theorem 1.4]{Mo96}, and Jiang, Liu and Savin \cite[Theorem 11.1]{JLS}.
\end{remark}

We introduce some notations and terminologies relevant to Young diagrams. The number of boxes in a Young diagram is called the size of the Young diagram.
For a Young diagram $\bfdd$, write $\bfdd^{\mathrm t}$ for the transpose of $\bfdd$. Also write $\bfdd^+$ for the Young diagram
obtained by adding one box to the first row. When $\bfdd$ is non-empty, write $\bfdd^-$ for the Young diagram obtained by removing  one box from the last row. Also recall the collapse operation of Gerstenhaber \cite[Lemma 6.3.3]{CM}: for a Young diagram $\bfdd$ of size $2n+1$, there is a unique largest Young diagram of type B of the same size dominated by $\bfdd$ (in the usual dominance order). This is called the $\mathrm B-$collapse of $\bfdd$. The $\mathrm C-$collapse and $\mathrm D-$collapse of a Young diagram $\bfdd$ of size $2n$ are defined similarly.

Following \cite{Weis} and in analogy with the Langlands dual, define $\check
\g:=\g=\s\p_{2n}(\BC)$ and call it the metaplectic dual of $\g$.
Now define the metaplectic Lusztig-Spaltenstein duality map:
\begin{equation}\label{tdls00}
\begin{array}{rcl}
   \tilde{\mathrm d}_{\mathrm{LS}}:\overline{\mathrm{Nil}}(\check \g)&\rightarrow & \overline{\mathrm{Nil}}^{\mathrm{sp}}(\o_{2n}(\BC)),\\
     \bfdd&\mapsto &  \textrm{the $\mathrm D-$collapse of  $\bfdd^{\mathrm t}$}.
     \end{array}
   \end{equation}

Define another map
\be\label{tdsp00}
\begin{array}{rcl}
  \tilde{ \mathrm d}_{\mathrm{SP}}:  \overline{\mathrm{Nil}}^{\mathrm{sp}}(\o_{2n}(\BC))&\rightarrow &\overline{\mathrm{Nil}}^{\mathrm{ms}}(\g^*),\\
  \bfdd&\mapsto & \textrm{the $\mathrm C-$collapse of $(\bfdd^+)^-$},
     \end{array}
\ee
where $\overline{\mathrm{Nil}}^{\mathrm{ms}}(\g^*)$ denotes the set of metaplectic special nilpotent orbits in $\overline{\mathrm{Nil}}(\g^*)$.

\begin{prop}\label{lemin1}
\begin{itemize}
\item [(a)]
The map $\tilde{\mathrm d}_{\mathrm{LS}}$ in \eqref{tdls00} is  well-defined, surjective, and order reversing.
\item [(b)]
The map $\tilde{ \mathrm d}_{\mathrm{SP}}$ in \eqref{tdsp00} is well-defined,  bijective, and order preserving.
\end{itemize}
\end{prop}

\begin{defn} \label{def:MBV} The metaplectic Barbasch-Vogan duality is defined to be the composition of \eqref{tdls00} and \eqref{tdsp00}:
\begin{equation}\label{MBV}
  \tilde{\mathrm d}_{\mathrm{BV}}:=  \tilde{ \mathrm d}_{\mathrm{SP}}\circ
  \tilde{\mathrm d}_{\mathrm{LS} }:
  \overline{\mathrm{Nil}}(\check \g)\rightarrow \overline{\mathrm{Nil}}^{\mathrm{ms}}(\g^*).
\end{equation}
\end{defn}

\begin{eg*}
\begin{itemize}
\item The metaplectic Barbasch-Vogan dual of $[2n]$ (principal nilpotent orbit) is $[2,1^{2n-2}]$ (minimal nilpotent orbit).
\item The metaplectic Barbasch-Vogan dual of $[1^{2n}]$ (zero orbit) is $[2n]$ (principal nilpotent orbit).
\end{itemize}
\end{eg*}

\begin{remark} The duality map of \Cref{def:MBV} appears in \cite[Section 1.4.2]{MoUnip}. One may also find its detailed description in \cite[Section 7]{MR}.
\end{remark}

Let $\check \CO\in  \overline{\mathrm{Nil}}(\check \g)$. Following Arthur and Barbasch-Vogan \cite[Section 5]{BVUni}, consider a semisimple element of $\check \g$ that equals half of the neutral element in any $\s\l_2$-triple representing $\check \CO$. By identifying the universal Cartan subalgebra $\check \t$ of $\check \g =\s\p_{2n}(\C)$ with $\t^{*}$ via half of the trace form on $\g=\s\p_{2n}(\C)$, it determines a character $\chi _{\check \CO}: \oZ(\g)\rightarrow \C$, which is explicitly described as follows.  For every  integer $a\geq 0$, write
\begin{equation}
  \label{eq:defrho}
  \rho(a):=\left\{ \begin{array}{ll}
                     (1, 2, \cdots, \frac{a-1}{2}), \quad &\textrm{if $a$ is odd;}\\
                     (\frac{1}{2}, \frac{3}{2}, \cdots, \frac{a-1}{2}), \quad &\textrm{if $a$ is even;}\\
                    \end{array}
                 \right.
\end{equation}
By convention, $\rho(1)$ and $\rho(0)$ are  the empty sequence.
Write $a_1\geq  a_2\geq \cdots\geq a_s>0$ for the rows of the Young diagram of $\check \CO$. Then the character $\chi_{\check \CO}$ is represented by the $\sfW_n-$orbit of the element
\begin{equation}\label{chico}
 (\rho( a_1), \rho(a_2),  \cdots, \rho(a_s), 0, 0, \cdots, 0 )\in \BC^n,
\end{equation}
Here the number of $0$'s is half of the number of odd rows in the Young diagram of $\check \CO$.

Recall the following basic result of Duflo \cite{Dix} \cite[Section 3]{Bor}, which is valid for any complex reductive Lie algebra $\g_1$: for every character $\chi$ of $\oZ(\g_1)$, there exists a unique maximal ideal of $\oU(\g_1)$ that contains the kernel of $\chi$, to be called
the maximal ideal of $\oU(\g_1)$ with infinitesimal character $\chi$.
Note that all maximal ideals of $\oU(\g_1)$ are primitive ideals.

Recall that $\check \g =\s\p_{2n}(\BC)$, the metaplectic dual of $\g =\s\p_{2n}(\BC)$.

\begin{introtheorem}\label{thm16}
 Let $\check \CO\in  \overline{\mathrm{Nil}}(\check \g)$ and denote by $I_{\check \CO}$ the maximal ideal of $\oU(\g)$ with infinitesimal character $\chi_{\check \CO}$. Then the associated variety of $I_{\check \CO}$ equals the closure of $\tilde{\mathrm d}_{\mathrm{BV}}(\check \CO)$.
\end{introtheorem}

\begin{remark} In the case of the Langlands dual, the corresponding
  result is a theorem of Barbasch-Vogan giving a  representation-theoretical interpretation of the
  Lusztig-Spaltenstein duality.
  This will be reviewed in \Cref{sec:review}.
  Our first proof of Theorem \ref{thm16} (\Cref{sec:proof1}) is reduced to this
  case via the technique of theta lifting (for complex dual pairs in
  the stable range).
\end{remark}

\begin{remark} The afore-mentioned result of Barbasch-Vogan is based on their work \cite{BVPri1,BVPri2} on the classification of primitive ideals in terms of Weyl group representations and Springer correspondence.
Our second
  proof (\Cref{sec:proof2}) is based on the
  interpretation of the Kazhdan-Lusztig conjectures for non-integral
  infinitesimal character given in \cite{VoIC4,ABV}, and the metaplectic
  dual map $\tilde{\mathrm d}_{\mathrm{BV}}$ is modeled by that in \cite{BVUni}.
\end{remark}

Write $\widetilde{\Sp}_{2n}(\R)$ for the metaplectic double cover of the real symplectic group $\Sp_{2n}(\R)$.
Recall that a  smooth Fr\'echet representation of  moderate growth  of a real reductive group is called a Casselman-Wallach representation \cite{Ca89,Wa2} if its Harish-Chandra module has  finite length. Following Barbasch and Vogan \cite{ABV,BVUni}, we make the following definition.

\begin{defn}
Let $\check \CO\in \overline{\mathrm{Nil}}(\check \g)$. We say that a genuine irreducible Casselman-Wallach representation $V$ of $\widetilde{\Sp}_{2n}(\R)$  is attached to $\check \CO$ if
\begin{itemize}
\item the infinitesimal character of $V$ equals $\chi_{\check \CO}$, and
\item the associated variety of the annihilator ideal of $V$ equals  the closure of $\tilde{\mathrm d}_{\mathrm{BV}}(\check \CO)$.
\end{itemize}
\end{defn}

We will call a genuine irreducible Casselman-Wallach representation $V$ of $\widetilde{\Sp}_{2n}(\R)$ (metaplectic) special unipotent if it is attached to $\check \CO$, for some $\check \CO\in \overline{\mathrm{Nil}}(\check \g)$. As examples, the irreducible pieces of the oscillator representations are (metaplectic) special, and are attached to the principal nilpotent orbit of $\check \g=\s\p_{2n}(\BC)$.

\vsp
The proof of our main results is not hard, and we will give two proofs, using representation theory of complex groups in both.
The first proof uses theta correspondence for complex dual pairs while the second one uses Kazhdan-Lusztig duality for complex groups. In the rest of this introduction, we will give some remarks on representation theory of the real metaplectic group as well as the relationship of this article to  the authors' series of two papers \cite{BMSZ1,BMSZ2}, in which we construct and classify special unipotent representations of real classical groups (including the real metaplectic group). As a consequence of the construction and classification, we show that all of them are unitarizable, as predicted by the Arthur-Barbasch-Vogan conjecture \cite[Introduction]{ABV}. Apart from their clear and well-known interest for the theory of automorphic forms \cite{ArUni,ArEnd}, special unipotent representations belong to a fundamental class of unitary representations to be constructed from nilpotent coadjoint orbits in the Kirillov philosophy (the orbit method; see \cite{Ki62,Ko70,VoBook}). In the case of complex groups, the
  notion of special unipotent representations was introduced in \cite{BVUni} where
  their characters and properties predicted by Arthur were
  established. For the complex classical groups they were shown to be
  unitary, and a larger class is introduced in \cite[Sections 4 and
    5]{B89} and shown that they form the building blocks for the unitary representations.
The representations termed (metaplectic) special in this paper are
  part of these building blocks.  They appear in Proposition
\ref{prop:sphu} in this paper. See also \cite{VoBook,Vo89}.
Special unipotent representations form the
building blocks of the spherical unitary dual of split real and p-adic
groups as detailed in \cite{B.Sph} and \cite{BC}. We remark that Losev, Mason-Brown and Matvieievskyi \cite{LMBM} have recently proposed a notion of unipotent representations for a complex reductive group, extending the notion of special unipotent representations (of Arthur and Barbasch-Vogan).

From their historical origin, representations of $\widetilde{\Sp}_{2n}(\R)$ have often been investigated in the framework of the oscillator representations and local theta correspondence \cite{Howe79,Howe89}. See for example \cite{AB2,GS}. Indeed our construction of special unipotent representations for all real classical groups in \cite{BMSZ2} (the second paper in the series) is carried out in this framework, and it requires us to have a suitable notion of nilpotent orbit duality for the real metaplectic group, which the current article provides.

In addition Renard and Trapa \cite{RT1} have established a Kazhdan-Lusztig algorithm for characters of irreducible genuine representations of $\widetilde{\Sp}_{2n}(\R)$ (with metaplectic integral infinitesimal character),
following the seminal work of Vogan \cite{VoIC3,VoIC4} on irreducible characters of real reductive linear groups (with integral infinitesimal character). We remark that the character theory of Vogan and Renard-Trapa form part of the ingredients towards the main goal of \cite{BMSZ1} (the first paper in the series), which is to count special unipotent representations attached to any $\check \CO$ in $\overline{\mathrm{Nil}}(\check \g)$.

\vsp

This article is organized as follows. In \Cref{sec:review}, we review the Barbasch-Vogan duality for classical Lie algebras. In \Cref{sec:MBV-BV}, we relate
the metaplectic Barbasch-Vogan duality with the Barbasch-Vogan duality for a (much) larger
symplectic Lie algebra. Along the way we prove basic properties of the
metaplectic Barbasch-Vogan duality, namely Proposition \ref{lemin1}. To work with primitive ideals of $\oU(\g)$ with infinitesimal character $\chi_{\check \CO}$, our main idea is to ``lift'' all that we do in $\g =\s\p_{2n}(\C)$ to $\h = \o_{2n+2a+1}(\BC)$, through theta lifting for the complex dual pair $(G,H)=(\Sp_{2n}(\BC), \oO_{2n+2a+1}(\BC))$ in the so-called stable range (where we fix an integer $a\geq n$). This is done in \Cref{sec:theta}, and in particular it allows us to link primitive
ideals of $\oU(\g)$ with a metaplectic integral infinitesimal character with primitive ideals of $\oU(\h)$ with an integral infinitesimal character. \Cref{sec:proof1} will be devoted to the proof of Theorems \ref{thm13} and \ref{thm16}. In \Cref{sec:proof2}, we give our second proof of Theorems \ref{thm13} and \ref{thm16} from the perspective of
Kazhdan-Lusztig theory \cite{KL}, and
interpret the metaplectic Barbasch-Vogan duality in terms of double cells of Weyl group representations.
\vsp

\noindent {\bf Acknowledgements}: The authors thank Xuhua He, Hiroyuki Ochiai and David Renard for their interest and comments on an earlier version of the article. A special thanks goes to Zhiwei Yun who brought to our attention a general framework \cite{LY}
relevant to the contents of this article. The authors would also like to thank the referees for their valuable comments and suggestions.

\vsp
D. Barbasch is supported by NSF grant, Award Number 2000254. J.-J. Ma is supported by the National Natural Science Foundation of China (Grant No. 11701364 and Grant No. 11971305) and  Xiamen University
Malaysia Research Fund (Grant No. XMUMRF/2022-C9/IMAT/0019).
B. Sun is supported by National Key R \& D Program of China (No. 2022YFA1005300 and 2020YFA0712600) and New Cornerstone Investigator Program.
 C.-B. Zhu is supported by MOE AcRF Tier 1 grant R-146-000-314-114, and
Provost’s Chair grant E-146-000-052-001 in NUS.

C.-B. Zhu is grateful to Max Planck Institute for Mathematics in Bonn, for its warm hospitality and conducive work environment, where he spent the academic year 2022/2023 as a visiting scientist.

\section{Review of the Barbasch-Vogan duality}\label{sec:review}

In this section, we review the Barbasch-Vogan duality for complex classical Lie algebras.

We will work with a pair $(\epsilon, V)$, where $\epsilon=\pm 1$, and $V$ is an $\epsilon$-symmetric complex bilinear space, i.e., a finite dimensional complex vector space equipped with a non-degenerate $\epsilon$-symmetric bilinear form. Write $G_V$ for the isometry group of $V$, $\g_V$  for the Lie algebra of $G_V$, $\overline{\mathrm{Nil}}(\g_{V})$ for the set of nilpotent $G_V$-orbits in $\g_{V}$.

The following terminology helps us to achieve some notational economy.

\begin{dfnl}  Two pairs $(\epsilon, V)$ and $(\check \epsilon, \check
  V)$ are said to be Langlands duals of each other if one of the followings holds:
\[
  \begin{aligned}
&\epsilon = \check \epsilon=1, &&\dim V=\dim\check V&&\text{ even, } && G_V=\oO(V),&& G_{\check V}=\oO(\check V);\\
&\epsilon=-1, \check \epsilon=1,&&\dim V=\dim\check V-1&&\text{ even, }&& G_V=\Sp(V),&& G_{\check V}=\oO(\check V);\\
&\epsilon=1, \check \epsilon=-1,&&\dim V=\dim\check V+1&&\text{ odd, }&& G_V=\oO(V),&& G_{\check V}=\Sp(\check V).
  \end{aligned}
\]
\end{dfnl}

\begin{remark} For comparison, we note that $(\epsilon, V)$ and $(\check \epsilon, \check V)$ are metaplectic duals of each other if
\[
  \textrm{$\epsilon=\check \epsilon=-1, \ $  and $\ \dim V=\dim \check V$}.
\]
In this case the group and its dual are the metaplectic group
$\widetilde{\Sp}(V).$
\end{remark}

\medskip
We give a combinatorial description of the duality in terms of Young
diagrams.
Fix a pair $(\check \epsilon, \check V)$. The Lusztig-Spaltenstein
duality map \cite{Spa} is
\be\label{dls}
\begin{array}{rcl}
   \mathrm d_{\mathrm{LS}}:  \overline{\mathrm{Nil}}(\g_{\check V})&\rightarrow &\overline{\mathrm{Nil}}(\g_{\check V}),\\
     \bfdd&\mapsto&  \textrm{the $(\mathrm B, \mathrm C $ or $\mathrm D)-$collapse of  $\bfdd^{\mathrm t}$}.
     \end{array}
\ee
It is order reversing and the image of this map is $\overline{\mathrm{Nil}}^{\mathrm{sp}}(\g_{\check V})$.

When $(\check \epsilon, \check V)$ is the Langlands
dual of $(\epsilon, V)$, define an order preserving bijection:
\be\label{dsp}
\begin{array}{rcl}
  \mathrm d_{\mathrm{SP}}:  \overline{\mathrm{Nil}}^{\mathrm{sp}}(\g_{\check V})&\rightarrow &\overline{\mathrm{Nil}}^{\mathrm{sp}}(\g_V^*),\smallskip\\
     \bfdd&\mapsto&
     \left\{
                \begin{array}{ll}
                  \textrm{the $\mathrm C-$collapse of $\bfdd^-$,} \quad &\textrm{if $\g_{\check V}$ has type B;}\\
                      \textrm{the $\mathrm B-$collapse of $\bfdd^+$,} \quad &\textrm{if $\g_{\check V}$ has type C;}\\
                      \bfdd,  \quad &\textrm{if $\g_{\check V}$ has type D.}\\
                    \end{array}
                    \right.
     \end{array}
\ee
The Barbasch-Vogan duality maps nilpotent orbits of $\mathfrak
g_{\check V}$ to special nilpotent orbits in $\mathfrak g_{V}^*$.
It is the composition of \eqref{dls} and \eqref{dsp}:
\begin{equation}\label{DefBV}
   \mathrm d_{\mathrm{BV}}:=    \mathrm d_{\mathrm{SP}}\circ \mathrm d_{\mathrm{LS} }: \overline{\mathrm{Nil}}(\g_{\check V})\rightarrow \overline{\mathrm{Nil}}^{\mathrm{sp}}(\g_{V}^*).
\end{equation}
\trivial[h]{
  Remark: $d_{SP}$ from type B to C is the inverse of $d_{SP}$ from type C to
  type B. See ``Sommers, Lusztig's Canonical Quotient and Generalized Duality,
  Section~10''.}

The group $G_V$ acts on $\oU(\g_V)$ through the Adjoint
representation. By the Harish-Chandra isomorphism (or a slight variant of the Harish-Chandra isomorphism), we have an identification
\[
 \oU(\g_V)^{G_V}=\left(\oS(\C^n)\right)^{\sfW_n}
\]
 as in \eqref{zg}, where $n=\lfloor\frac{\dim V}{2}\rfloor$.

\trivial[h]{The invariant space $\oU(\g_V)^{G_V}$ equals the
center $\oZ(\g_V)$ of $\oU(\g_V)$ unless $G_V$ is an even orthogonal
group {\clrr when it is ??}. In all cases, we have an identification
\[
 \oU(\g_V)^{G_V}=\left(\oS(\C^n)\right)^{\sfW_n}
\]
 as in \eqref{zg}, where $n=\lfloor\frac{\dim V}{2}\rfloor$.
 }

Let $\check \CO\in \overline{\mathrm{Nil}}(\g_{\check V})$. As in \eqref{chico}, we attach the following algebraic character $\chi _{\check \CO}$ of $\oU(\g_V)^{G_V}$:
\begin{equation}\label{usual-chico}
 \chi _{\check \CO}:= (\rho( a_1), \rho(a_2),  \cdots, \rho(a_s), 0, 0, \cdots, 0 ),
\end{equation}
where $a_1\geq  a_2\geq \cdots\geq a_s>0$ are the rows of the Young
diagram of $\check \CO$. This is the usual Arthur infinitesimal
character, determined by $\half {^{L}h}$ in the
  notation of
  \cite[Section 5]{BVUni}.

The following result \cite[Corollary A3]{BVUni} gives a
representation-theoretical interpretation of the Barbasch-Vogan duality map, in terms of maximal ideals. More
discussions of Lusztig-Spaltenstein and Barbasch-Vogan dualities can
be found in \cite[Section 3.5]{Ach}.

\begin{thm}\label{DesBV} {\bf {\upshape (Barbasch-Vogan)}}
Let $\check \CO\in  \overline{\mathrm{Nil}}(\g_{\check V})$ and denote
by $I_{\check \CO}$ the maximal $G_V$-stable ideal of $\oU(\g_V)$ that
contains the kernel of $\chi_{\check \CO}$. Then the associated
variety of $I_{\check \CO}$ equals the closure of ${\mathrm
  d}_{\mathrm{BV}}(\check \CO)$.
\end{thm}

\delete{
\begin{remark} Except for the case when $G_V$ is an even orthogonal
  group, all ideals of $\oU(\g_V)$ are $G_V-$stable. When $G_V$ is an
  even orthogonal group, a maximal $G_V-$stable ideal of $\oU(\g_V)$
  is either a maximal ideal of $\oU(\g_V)$ or the intersection of two
  distinct maximal ideals of $\oU(\g_V)$.
\end{remark}
}

\section{Metaplectic Barbasch-Vogan duality and Barbasch-Vogan duality}
\label{sec:MBV-BV}

Recall the Lusztig-Spaltenstein duality for the symplectic Lie algebra:
\[
\begin{array}{rcl}
   \mathrm d_{\mathrm{LS}}:  \overline{\mathrm{Nil}}(\s\p_{2n}(\BC))&\rightarrow &\overline{\mathrm{Nil}}^{\mathrm{sp}}(\s\p_{2n}(\BC)),\\
     \bfdd&\mapsto&  \textrm{the $\mathrm C-$collapse of  $\bfdd^{\mathrm t}$}.
     \end{array}
     \]

We shall first relate the metaplectic Lusztig-Spaltenstein duality map $\tilde{\mathrm d}_{\mathrm{LS}}$ in \eqref{MBV} with the Lusztig-Spaltenstein duality map ${\mathrm d}_{\mathrm{LS}}$ for a larger symplectic Lie algebra $\s\p_{2n+2a}(\BC)$.

We start with some notations. For a Young diagram $\bfdd$, write $\nabla(\bfdd)$ for the Young diagram obtained from $\bfdd$ by removing the first (or the largest) column, and write $\check \nabla(\bfdd)$ for the Young diagram obtained from $\bfdd$ by removing the first (or the largest) row.   Write $\mathrm r_1(\bfdd)$ and $\mathrm c_1(\bfdd)$ for the lengths of the first row and the first column of $\bfdd$, respectively. Similar notations apply for a classical nilpotent orbit.

\begin{lem}\label{spspo000}
Suppose that $a\geq n$ is an integer. Then the diagram
\be\label{cdls}
 \begin{CD}
            \{\check \CO\in \overline{\mathrm{Nil}}(\s\p_{2n+2a}(\BC))\mid \mathrm{r}_1(\check \CO)=2a\} @>  \check \nabla  >> \overline{\mathrm{Nil}}(\s\p_{2n}(\BC)) \\
            @V \bfdd\mapsto \textrm{the $\mathrm C-$collapse of  $\bfdd^{\mathrm t}$} VV           @V\bfdd\mapsto \textrm{the $\mathrm D-$collapse of  $\bfdd^{\mathrm t}$}V V\\
               \{\CO\in \overline{\mathrm{Nil}}(\s\p_{2n+2a}(\BC))\mid \mathrm{c}_1( \CO)=2a\}
             @> \nabla   >>  \overline{\mathrm{Nil}}(\o_{2n}(\BC))\\
  \end{CD}
\ee
commutes.
\end{lem}

\begin{proof} It is easy to check that the diagrams
\[
 \begin{CD}
            \{\check \CO\in \overline{\mathrm{Nil}}(\s\p_{2n+2a}(\BC))\mid \mathrm{r}_1(\check \CO)=2a\} @>  \check \nabla  >> \overline{\mathrm{Nil}}(\s\p_{2n}(\BC)) \\
            @V \bfdd\mapsto \textrm{$\bfdd^{\mathrm t}$} VV           @V\bfdd\mapsto \textrm{$\bfdd^{\mathrm t}$}V V\\
               \{\CO\in \overline{\mathrm{Nil}}(\g\l_{2n+2a}(\BC))\mid \mathrm{c}_1(\CO)=2a\}
             @> \nabla   >>  \overline{\mathrm{Nil}}(\g\l_{2n}(\BC))\\
  \end{CD}
\]
and
\[
 \begin{CD}
            \{ \CO\in \overline{\mathrm{Nil}}(\g\l_{2n+2a}(\BC))\mid \mathrm{c}_1( \CO)=2a\} @>   \nabla  >> \overline{\mathrm{Nil}}(\g\l_{2n}(\BC)) \\
            @V \textrm{$\mathrm C-$collapse} VV           @V \textrm{$\mathrm D-$collapse}V V\\
               \{\CO\in \overline{\mathrm{Nil}}(\s\p_{2n+2a}(\BC))\mid \mathrm{c}_1( \CO)=2a\}
             @> \nabla   >>  \overline{\mathrm{Nil}}(\o_{2n}(\BC))\\
  \end{CD}
\]
commute. Thus the diagram
\eqref{cdls} also commutes.
\end{proof}

We will use $\preccurlyeq$ to indicate the dominance order of the Young diagrams.

\begin{lem}\label{aaaa}
Suppose that $a\geq n$ is an integer. Let $\bfdd $ be a Young diagram of size $2n+2a$, and write $\bfdd_{\mathrm C}$ for its $\mathrm C-$collapse. If $\mathrm c_1(\bfdd_{\mathrm C})=2a$, then $\mathrm c_1(\bfdd)=2a$.

\end{lem}
\begin{proof}
If $\mathrm c_1(\bfdd)>2a$, then it is clear that $\mathrm c_1(\bfdd_{\mathrm C})>2a$, which is a contradiction. Thus  $\mathrm c_1(\bfdd)\leq 2a$. Suppose that $\mathrm c_1(\bfdd)\leq 2a-1$. Let $\bfdd_1$ be the Young diagram  such that its columns are
\[
  \left\{
    \begin{array}{ll}
      (2a-1, 2n+1), & \hbox{if $a>n$;} \\
      (2a-1,2a-1,2), & \hbox{if $a=n\geq 2$;} \\
      (2a-1,2a-1,2a-1,2a-1), & \hbox{if $a=n=1$.}
    \end{array}
  \right.
\]
Then $\bfdd_1$ has type C, and  $\bfdd_1\preccurlyeq \bfdd$. Hence $\bfdd_1\preccurlyeq \bfdd_{\mathrm C}$. This implies that $2a-1=\mathrm c_1(\bfdd_1)\geq \mathrm c_1(\bfdd_{\mathrm C})=2a$, which is also a contradiction. This proves the lemma.
\end{proof}

Part (a) of Proposition \ref{lemin1} may be reformulated as the following lemma.
\begin{lem}\label{lemin11}
The right vertical arrow of \eqref{cdls} is order reversing, and its image equals $\overline{\mathrm{Nil}}^{\mathrm{sp}}(\o_{2n}(\BC))$.
\end{lem}
\begin{proof}
We only prove the second assertion as the first one is obvious. Note
that the two horizontal arrows of \eqref{cdls} are
  bijections, and the bottom horizontal arrow yields a bijection
\[
   \mathrm N_1:= \{\CO\in \overline{\mathrm{Nil}}^{\mathrm{sp}}(\s\p_{2n+2a}(\BC))\mid \mathrm{c}_1( \CO)=2a\}
             \xrightarrow{ \nabla}  \overline{\mathrm{Nil}}^{\mathrm{sp}}(\o_{2n}(\BC)).
\]
Thus it remains to show that the image of the left vertical arrow of  \eqref{cdls} equals $\mathrm N_1$. This follows from Lemma \ref{aaaa} and the surjectivity of the Lusztig-Spaltenstein duality map for symplectic Lie algebras. 
\end{proof}

Recall we have the order preserving bijective map
\[
\begin{array}{rcl}
  \mathrm d_{\mathrm{SP}}:  \overline{\mathrm{Nil}}^{\mathrm{sp}}(\s\p_{2n}(\BC))&\rightarrow &\overline{\mathrm{Nil}}^{\mathrm{sp}}(\o_{2n+1}(\BC)),\\
     \bfdd & \mapsto &           \textrm{the $\mathrm B-$collapse of $\bfdd^+$}
     \end{array}
\]
The Barbasch-Vogan duality for an odd orthogonal group is given by
     \[
   \mathrm d_{\mathrm{BV}}=:    \mathrm d_{\mathrm{SP}}\circ \mathrm d_{\mathrm{LS} }: \overline{\mathrm{Nil}}(\s\p_{2n}(\BC))\rightarrow \overline{\mathrm{Nil}}^{\mathrm{sp}}(\o_{2n+1}(\BC)).
\]

We now relate the map $\tilde{ \mathrm d}_{\mathrm{SP}}$ in \eqref{tdsp00} with the map $\mathrm d_{\mathrm{SP}}$ for the larger symplectic Lie algebra $\s\p_{2n+2a}(\BC)$.
We start with an elementary lemma.

\begin{lem}\label{colb}
Let $\bfdd_1, \bfdd_2$ be two Young diagrams of size $2n+1$. Suppose that $\mathrm{c}_1(\bfdd_1)$ is odd, and  $\nabla(\bfdd_1)=(\nabla(\bfdd_2))^-$.   Then $\bfdd_1$ and $\bfdd_2$ have the same $\mathrm B-$collapses.
\end{lem}
\begin{proof}
Note that $\bfdd_1$ is the largest Young diagram such that its first column has odd length, and  $\bfdd_1\preccurlyeq \bfdd_2$. The lemma then follows by noting that the first column of every Young diagram of type B has odd length.
\end{proof}

\begin{lem}\label{spspo00}
Suppose that $a\geq n$ is an integer. Then the diagram
\be\label{cdls22}
 \begin{CD}
                          \{\CO\in \overline{\mathrm{Nil}}(\s\p_{2n+2a}(\BC))\mid \mathrm{c}_1( \CO)=2a\}
             @> \nabla   >>  \overline{\mathrm{Nil}}(\o_{2n}(\BC))\\
            @V \bfdd\mapsto \textrm{the $\mathrm B-$collapse of  $\bfdd^+$} VV           @V\bfdd\mapsto \textrm{the $\mathrm C-$collapse of  $(\bfdd^+)^-$}V V\\
 \{\CO\in \overline{\mathrm{Nil}}(\o_{2n+2a+1}(\BC))\mid \mathrm{c}_1(\CO)=2a+1\} @>  \nabla  >> \overline{\mathrm{Nil}}(\s\p_{2n}(\BC)) \\
  \end{CD}
\ee
commutes.
\end{lem}
\begin{proof}
Suppose that $\bfdd\in \overline{\mathrm{Nil}}(\s\p_{2n+2a}(\BC))$ and $\mathrm c_1(\bfdd)=2a$. If $a=0$ then the lemma obviously holds. Thus we assume that $a>0$ so that $\bfdd$ is not the empty Young diagram.

Let $\bfdd'$ denote the Young diagram such that  $\mathrm{c}_1(\bfdd')=2a+1$ and
\be\label{bdff1}
\nabla(\bfdd')=(\nabla(\bfdd^+))^-.
\ee
 By  Lemma \ref{colb},
\be\label{bdff2}
  \textrm{the $\mathrm B-$collapse of  $\bfdd^+$}= \textrm{the $\mathrm B-$collapse of  $\bfdd'$}.
\ee
Note that
\be\label{bdff3}
  \nabla(\textrm{the $\mathrm B-$collapse of  $\bfdd'$})=\textrm{the $\mathrm C-$collapse of  $\nabla(\bfdd')$}
\ee
and
\be\label{bdff4}
  \nabla(\bfdd^+)=(\nabla(\bfdd))^+.
\ee
The lemma then follows by combining \eqref{bdff1}, \eqref{bdff2}, \eqref{bdff3} and \eqref{bdff4}.
\end{proof}

\begin{lem}\label{spspo}
The left vertical arrow of \eqref{cdls22} yields a bijective map
 \[
                      \{\CO\in \overline{\mathrm{Nil}}^{\mathrm{sp}}(\s\p_{2n+2a}(\BC))\mid \mathrm{c}_1( \CO)=2a\}\rightarrow
 \{\CO\in \overline{\mathrm{Nil}}^{\mathrm{sp}}(\o_{2n+2a+1}(\BC))\mid \mathrm{c}_1(\CO)=2a+1\}.
\]

\end{lem}
\begin{proof}
Recall that the map
\be\label{mapsp}
\begin{array}{rcl}
  \mathrm d_{\mathrm{SP}}:  \overline{\mathrm{Nil}}^{\mathrm{sp}}(\s\p_{2n+2a}(\BC))&\rightarrow &\overline{\mathrm{Nil}}^{\mathrm{sp}}(\o_{2n+2a+1}(\BC)),\\
     \bfdd & \mapsto &           \textrm{the $\mathrm B-$collapse of $\bfdd^+$.}
     \end{array}
\ee
is well-defined and bijective. Thus it suffices to show that if $\bfdd\in \overline{\mathrm{Nil}}^{\mathrm{sp}}(\s\p_{2n+2a}(\BC))$ and $\mathrm{c}_1(\mathrm d_{\mathrm{SP}}(\bfdd))=2a+1$, then $\mathrm{c}_1(\bfdd)=2a$.

It is clear that $\mathrm c_1(\bfdd)\leq 2a+1$.
Suppose by contradiction that  $\mathrm c_1(\bfdd)= 2a+1$. Since $\bfdd$ is special, we know that the first two columns of $\bfdd$ both have length $2a+1$. This contradicts the assumption that $a\geq n$. Therefore we conclude that $\mathrm c_1(\bfdd)\leq 2a$.

 Suppose that $\mathrm c_1(\bfdd)\leq 2a-1$. Let $\bfdd_1$ be the Young diagram  such that its columns are
\[
  \left\{
    \begin{array}{ll}
      (2a-1, 2n+2), & \hbox{if $a\geq n+2$;} \\
      (2a-1,2a-1,1), & \hbox{if $a=n+1$;} \\
      (2a-1,2a-1,3), & \hbox{if $a=n\geq 2$;}\\
         (2a-1,2a-1,1,1,1), & \hbox{if $a=n=1$.}
    \end{array}
  \right.
\]
Then $\bfdd_1$ has type B, and  $\bfdd_1\preccurlyeq \bfdd^+$. Hence $\bfdd_1\preccurlyeq (\bfdd^+)_B$. This implies that $2a-1=\mathrm c_1(\bfdd_1)\geq \mathrm c_1((\bfdd^+)_B)=2a+1$, which is also a contradiction. This proves the lemma.
\end{proof}

Part (b) of Proposition \ref{lemin1}
may be reformulated as the following lemma.
\begin{lem}\label{lemin111}
The right vertical arrow of \eqref{cdls22} induces an order preserving bijection map from  $\overline{\mathrm{Nil}}^{\mathrm{sp}}(\o_{2n}(\BC))$ onto $\overline{\mathrm{Nil}}^{\mathrm{ms}}(\s\p_{2n}(\BC))$.
\end{lem}
\begin{proof}
Note that the top horizontal arrow of \eqref{cdls22} induces a bijection
 \[
                      \{\CO\in \overline{\mathrm{Nil}}^{\mathrm{sp}}(\s\p_{2n+2a}(\BC))\mid \mathrm{c}_1( \CO)=2a\}\rightarrow
 \overline{\mathrm{Nil}}^{\mathrm{sp}}(\o_{2n}(\BC)),
\]
and the bottom
horizontal arrow of \eqref{cdls22} induces a bijection
 \[
                      \{\CO\in \overline{\mathrm{Nil}}^{\mathrm{sp}}(\o_{2n+2a+1}(\BC))\mid \mathrm{c}_1( \CO)=2a+1\}\rightarrow
\overline{\mathrm{Nil}}^{\mathrm{ms}}(\s\p_{2n}(\BC)).
\]
Thus the lemma follows from Lemmas \ref{spspo00} and \ref{spspo}.
\end{proof}

Combining Lemma \ref{spspo000} and Lemma \ref{spspo00}, we arrive at the following result.

\begin{prop}\label{cdnn}
Suppose that $a\geq n$ is an integer. Then the diagram
\[
 \begin{CD}
            \{\check \CO\in \overline{\mathrm{Nil}}(\s\p_{2n+2a}(\BC))\mid \mathrm{r}_1(\check \CO)=2a\} @>  \check \nabla  >> \overline{\mathrm{Nil}}(\s\p_{2n}(\BC)) \\
            @V \mathrm d_{\mathrm{BV}} VV           @V \tilde{\mathrm d}_{\mathrm{BV}} V V\\
               \{\CO\in \overline{\mathrm{Nil}}^{\mathrm{sp}}(\o_{2n+2a+1}(\BC))\mid \mathrm{c}_1( \CO)=2a+1\}
             @> \nabla   >>  \overline{\mathrm{Nil}}^{\mathrm{ms}}(\s\p_{2n}(\BC))\\
  \end{CD}
\]
commutes.
\end{prop}

\section{Theta lifting for the complex dual pair $(\Sp_{2n}(\BC), \oO_{2n+2a+1}(\BC))$}
\label{sec:theta}

Theta lifting for complex dual pairs is known explicitly by the work of Adams and Barbasch \cite{AB}. See also \cite{B17}. In this section we will examine a specific complex dual pair and establish some qualitative properties of the theta lifting
which will be used in the proof of Theorems \ref{thm13} and \ref{thm16} in \Cref{sec:proof1}.

\delete{
{\clrr The material is in much more detail in \cite{AB}, and some in \cite{B89}. At least a
  mention would be in order}.
  }

Consider the complex dual pair $(G,H):=(\Sp_{2n}(\BC), \oO_{2n+2a+1}(\BC))$ in $\Sp_{2n(2n+2a+1)}(\BC)$ \cite{Howe79}, where we fix an integer $a\geq n$.
The condition $a\geq n$ ensures that the dual pair $(G,H)$ is in the so-called stable range with $G$ the smaller member \cite{Li89}.
All our assertions concerning $n$ appearing in this note are trivial when $n=0$. We thus assume that $n\geqq 1$.

Fix a Cartan involution $\tau$ of $\Sp_{2n(2n+2a+1)}(\BC)$ that stabilizes both $G$ and $H$. Then the fixed point sets $\Sp_{2n(2n+2a+1)}(\BC)^\tau$, $G^\tau$ and $H^\tau$ are respectively maximal compact subgroups of $\Sp_{2n(2n+2a+1)}(\BC)$, $G$ and $H$.

As before, $\g=\s\p_{2n}(\BC)$ is the Lie algebra of $G$, and $\check
\g=\g$ is the metaplectic dual of $\g$. Denote by $\tau
:\g\rightarrow \g$  the differential of $\tau: G\rightarrow G$,
which is the complex conjugation of $\g$ with respect to a compact real form.
We identify the complexified Lie algebra $\g_\BC:=\g\otimes_\R \BC$ with $\g\times \g$ via the complexification map
\begin{equation}
\label{eq:comp}
  \g\rightarrow \g\times \g, \quad x\mapsto (x,  \tau(x)).
\end{equation}
We also identify the complexification of $G^\tau$ with $G$ via the inclusion homomorphism $G^\tau\rightarrow G$. Then every $(\g_\BC, G^\tau)-$module  is also a $(\g\times \g, G)-$module. Here a  $(\g\times \g, G)-$module means a $\g\times \g-$module carrying a locally algebraic representation of $G$ with the usual compatibility conditions.
Similarly, denote by $\h$ the Lie algebra of $H$. Then every $(\h_\BC,
H^\tau)-$module is also a $(\h\times \h, H)-$module.

We introduce the algebraic theta lifting. Let $\omega^\infty$ be the smooth oscillator representation of $\Sp_{2n(2n+2a+1)}(\BC)$ (\cite{Weil, Howe89}). Write $\sY$ for the set of $\Sp_{2n(2n+2a+1)}(\BC)^\tau-$finite vectors in $\omega^\infty$.
By restriction, $\sY$ is a $(\g\times \g, G)\times (\h\times \h,H)-$module. For every $(\g\times \g, G)-$module $V$, define an $(\h\times \h, H)-$module
\[
  \Theta(V):=(\sY \otimes \check V)_{\g\times \g},
\]
where $\check V$ denotes the contragredient of $V$, and a subscript Lie algebra indicates the coinvariant space.

\begin{lem}\label{theta1}
Suppose that $V$ is an irreducible $(\g\times \g, G)-$module. Then $\Theta(V)$ is a nonzero $(\h\times \h, H)-$module of finite length.
\end{lem}

\begin{proof} It is a general result of Howe \cite{Howe89} that $\Theta(V)$ is a $(\h\times \h, H)-$module of finite length. The fact that $\Theta(V)$ is nonzero is a consequence of the stable range condition
\cite{PrPr}.
\end{proof}

Write $\oZ(\h)$ for the center of the universal enveloping algebra $\oU(\h)$ of $\h$.  By Harish-Chandra isomorphism, we have an identification
\[
 \oZ(\h)=\left(\oS(\C^{n+a})\right)^{\sfW_{n+a}},
\]
where $\sfW_{n+a}$ is the Weyl group of type $\mathrm B_{n+a}$ or $\mathrm C_{n+a}$ (defined in \Cref{sec:intro}). The set of (algebraic) characters of $\oZ(\h)$ is thus identified with the set of $\sfW_{n+a}$-orbits in $\BC^{n+a}$.

From the correspondence of infinitesimal characters \cite{PrzInf}, we have the following

\begin{lem}\label{theta2}
Suppose that $V$ is an irreducible $(\g\times \g, G)-$module with infinitesimal character
\[
  (\lambda_1, \lambda_2, \cdots, \lambda_n; \lambda'_1, \lambda'_2, \cdots, \lambda'_n)\in \BC^n\times \BC^n.
\]
Then the $(\h\times \h, H)-$module  $\Theta(V)$ has infinitesimal character
\[
  (\lambda_1, \lambda_2, \cdots, \lambda_n, \frac{1}{2}, \frac{3}{2}, \cdots, \frac{2a-1}{2} ; \lambda'_1, \lambda'_2, \cdots, \lambda'_n, \frac{1}{2}, \frac{3}{2}, \cdots, \frac{2a-1}{2})\in \BC^{n+a}\times \BC^{n+a}.
\]
\end{lem}

We will use result of the following proposition systematically.

\begin{prop}[\cite{J85} and  \cite{Vo89}] \label{p:orbit}
For every irreducible $(\g\times \g, G)-$module $V$, there is a
unique nilpotent orbit $\CO\in \overline{\mathrm{Nil}}(\g^*)$ such that the associated variety of
the annihilator ideal of $V$ equals $\overline \CO\times \overline \CO$, where
the overbar indicates the closure.
\end{prop}

In the notation of Proposition \ref{p:orbit}, we write $\mathrm{Orbit}(V)$
for the nilpotent orbit $\CO$. The same notation applies to irreducible
$(\h\times \h, H)-$modules.

\begin{lem}\label{theta3}
  Suppose that $V$ is an irreducible $(\g\times \g, G)-$module. Then there exists an irreducible subquotient $\eta(V)$ of $\Theta(V)$ whose orbit satisfies the following conditions:
 \[
  {\rcc_1(\orb(\eta(V)) = 2a+1 } \quad and \quad \nabla(\mathrm{Orbit}(\eta(V)))=\mathrm{Orbit}(V).
\]
\end{lem}

\begin{proof} We know from a result of Loke and Ma \cite[Theorem
  D]{LM}, combined with a general result of Vogan \cite[Theorem
  8.4]{Vo89}, that the associated variety of the annihilator ideal of
  $\Theta (V)$ is irreducible. Denote its open orbit by
  $\mathrm{Orbit}(\Theta(V))$. From the explicit description of
  nilpotent orbit correspondence in \cite{DKPC},
  $\mathrm{Orbit}(\Theta(V))$ satisfies the two stated conditions of
  the current lemma.
We then pick an irreducible quotient $\eta(V)$ of $\Theta(V)$ whose
Gelfand-Kirillov dimension is that of $\Theta(V)$. Obviously $\mathrm{Orbit}(\eta(V))=\mathrm{Orbit}(\Theta(V))$. This proves
the lemma.
\end{proof}

For every ideal $I$ of $\oU(\g)$, $\oU(\g)/I$ is a $(\g\times \g, G)-$module such that
\[
  (X, Y)\cdot v=Xv-vY, \quad \textrm{for all } \, X, Y\in \g, \, v\in \oU(\g)/I.
\]
This module is generated by a spherical vector (the image of the
vector $1$).
Note that $I$ is maximal if and
only if $\oU(\g)/I$ is irreducible. Furthermore every spherical
irreducible $(\g\times \g, G)$-module is isomorphic to $\oU(\g)/I$ for a unique
maximal ideal $I$ of $\oU(\g)$. See \cite[Section 9.6]{Dix}.

\begin{prop}\label{theta4}
Suppose that $V$ is a spherical irreducible unitarizable $(\g\times \g, G)-$module. Then $\Theta(V)$ is also a spherical irreducible unitarizable $(\h\times \h, H)-$module.
\end{prop}

\begin{proof} The fact that $\Theta(V)$ is spherical follows from the correspondence of $G^{\sigma}$ and $H^{\sigma}$ types in the space of
joint harmonics \cite[Section 3]{Howe89}. The irreducibility is in
\cite[Theorem~A]{LM}, and the unitarity is in \cite{Li89}, all because
of the stable range condition.
\end{proof}

\medskip

For each nilpotent orbit $\ckcO\in \overline{\mathrm{Nil}}(\check \fgg)$,
let $V_{\ckcO}$ be the spherical irreducible $(\fgg\times \fgg, G)-$module with
infinitesimal character $(\chico, \chico)$.
We will need the following key
property of $V_{\ckcO}$ from \cite{B89}.

\begin{prop} \label{prop:sphu}
  The $(\fgg\times \fgg, G)-$module $V_{\ckcO}$ is unitarizable.
\end{prop}
\begin{proof}
  We pair the rows in the Young diagram of $\ckcO$ by their length. Without loss
  of generality, we can assume that
  $\ckcO$ has row lengths as a multi-set (set with multiplicities)
  \[
    \set{a_{1}, a_2,\cdots, a_{2p-1},a_{2p}, b_1, b_1, b_2, b_2, \cdots, b_l,
      b_l| a_i \in 2\bN, a_i < a_{i+1}, b_i \in \bN}.
  \]
  Note that $a_1< \cdots< a_{2p}$ are distinct even integers and we allow $a_1=0$.
  Let $a := \sum_{i=1}^{2p} a_i$,
  $\ckfgg_a:=\fgg_a := \fsp_{2a}(\bC)$, $G_a = \Sp_{2a}(\bC)$,
  and $\ckcO_0$ be the nilpotent orbit in
  $\overline{\mathrm{Nil}}(\ckfgg_a)$ whose Young diagram has rows $a_1<a_2< \cdots<
  a_{2p}$.

  The spherical $(\fgg_a\times \fgg_a, G_a)$-module $V_{\ckcO_0}$ has infinitesimal character $(\chi_{\ckcO_0},\chi_{\ckcO_0})$, where
  $\ckcO_0$ has row lengths
  \[
    \set{a_{1}, a_2,\cdots, a_{2p-1},a_{2p}},
  \]
  and $\chi_{\ckcO_0}$ is defined similarly as in \eqref{chico}.
  The unitarity of $V_{\ckcO_0}$ is a special
case of \cite[Proposition~10.6]{B89}.

  Now we conclude that $V_\ckcO$ is unitarizable since it is the spherical component of the
  unitary induction
  \[
  \Ind^{\Sp_{2n}(\bC)}_{P} (V_{\ckcO_0}\otimes \bfone\otimes \cdots\otimes  \bfone),
  \]
  where $P$ has Levi factor $\Sp_{2a}(\bC)\times \GL_{b_1}(\bC) \times
  \GL_{b_2}(\bC) \times \cdots \times \GL_{b_l}(\bC)$ and
  $\bfone$ denotes the trivial representations on
  $\GL_{b_i}(\bC)$ factors. (In fact the induced module is irreducible \cite[Sections 4 and 5]{B89}.)
  \end{proof}

\section{Proof of Theorems \ref{thm13} and \ref{thm16}}
\label{sec:proof1}

\subsection{Proof of Theorem \ref{thm13}}
We are in the setting of Theorem \ref{thm13}. Recall that $I$ is a primitive ideal of $\oU(\g)$ with a
metaplectic integral infinitesimal character $\chi_{I}$
(\Cref{def:mpli}), i.e. its coordinates are in
$(\frac{1}{2}+\BZ)^n$. Denote by $\CO\in
\overline{\mathrm{Nil}}(\g^*)$ the nilpotent orbit such that the
associated variety of $I$ equals the closure of $\CO$ (see \Cref{p:orbit}).

\begin{lem}\label{exv}
There exists an irreducible  $(\g\times \g, G)-$module $V$ such that its infinitesimal character is $(\chi_{I},\chi_{I})$,
and $\mathrm{Orbit}(V)=\CO$.
\end{lem}
\begin{proof}
Put $V_I:=\oU(\g)/I$, viewed as an admissible $(\g\times \g, G)-$module.
Then $V_I$ has finite length and infinitesimal character $(\chi_{I},\chi_{I})$.
Its associated variety equals $\overline
\CO\times \overline \CO$. Let $V$ be an irreducible subquotient of
$V_I$ with the largest Gelfand-Kirillov dimension. Its
  annihilator $J\otimes\oU(\g)+\oU(\g)\otimes \check J$ satisfies $I\subset J$, and the Gelfand-Kirillov dimensions are
  equal.  Here $\check J$ denotes the image of $J$ under the principal anti-automorphism of $\oU(\g)$ (which is $-1$ on $\g$).
  We conclude that $J=I$ by \cite[Korollar 3.6]{BK},
  and the associated variety of the annihilator
  ideal of $V$ as a $(\mathfrak g\times\mathfrak g,G)-$module also equals $\overline \CO\times \overline \CO$. This proves the lemma.
 \end{proof}

 Let $V$ be as in Lemma \ref{exv}, and let $\eta(V)$ be an irreducible subquotient of $\Theta (V)$, as in Lemma
 \ref{theta3}. By Lemmas \ref{theta2}, $\eta(V)$ has an integral
 infinitesimal character. By  \cite[Definition
 1.10]{BVUni} and the remarks immediately after, we conclude that $\mathrm{Orbit}(\eta(V))$ is
special, i.e. $\mathrm{Orbit}(\eta(V))\in
 \overline{\mathrm{Nil}}^{\mathrm{sp}}(\h^*)$. By Lemma
 \ref{theta3},  $\mathrm c_1(\orb(\eta(V)))=2a+1$ and $\nabla(\mathrm{Orbit}(\eta(V)))= \mathrm{Orbit}(V)$, which is $\CO$.
Therefore $\CO$ is metaplectic special by Proposition \ref{cdnn}. This completes the proof of Theorem \ref{thm13}.

\subsection{Proof of Theorem \ref{thm16}}

We are in the setting of Theorem \ref{thm16}. Recall that
$\check \CO\in \overline{\mathrm{Nil}}(\check \g)$ and $I_{\check \CO}$ is the maximal
ideal of $\oU(\g)$ with infinitesimal character $\chi_{\check \CO}$.
Then $V_{\check \CO}=\oU(\g)/I_{\check \CO}$ is the spherical irreducible
$(\g\times \g, G)-$module with infinitesimal character $(\chico, \chico)$.
By Proposition~\ref{prop:sphu}, $V_{\ckcO}$ is unitarizable.

Let $\check \CO'$ be the element in $\overline{\mathrm{Nil}}(\s\p_{2n+2a}(\BC))$ such that $\mathrm{r}_1(\check \CO')=2a$ and $\check \nabla(\check \CO')=\check \CO$. The Lie algebra  $\s\p_{2n+2a}(\BC)$ is the  Langlands dual of $\h=\o_{2n+2a+1}(\BC)$, and $\check \CO'$ determines a character $\chi_{\check \CO'}: \oZ(\h)\rightarrow \BC$ in the usual way (see \eqref{usual-chico}). Suppose that $\chi_{\check \CO}$ is represented by  $(\lambda_1, \lambda_2, \cdots, \lambda_n)$, then  $\chi_{\check \CO'}$ is represented by
\[
  (\lambda_1, \lambda_2, \cdots, \lambda_n, \frac{1}{2}, \frac{3}{2}, \cdots, \frac{2a-1}{2}).
\]

By Lemma \ref{theta2} and Proposition \ref{theta4}, $\Theta(V_{\check \CO})$ is a spherical irreducible unitarizable $(\h\times \h, H)$-module with the infinitesimal character $(\chi_{\check \CO'},\chi_{\check \CO'})$.
Hence
\[
  \Theta(V_{\check \CO})\cong \oU(\h)/I'
\]
where $I'$ is the maximal ideal of $\oU(\h)$ with infinitesimal
character $\chi_{\check \CO'}$. By Theorem \ref{DesBV}, we
know that the associated variety of $I'$ equals the closure of $\mathrm d_{\mathrm{BV}}(\check \CO')$, namely
\be\label{orbitt} \mathrm{Orbit}(\Theta(V_{\check \CO}))=\mathrm d_{\mathrm{BV}}(\check
\CO').  \ee Finally, we see that the associated variety of $I_{\check \CO}$ is the closure of
\begin{eqnarray*}
   && \mathrm{Orbit}(V_{\check \CO}) \\
   &=& \nabla(\mathrm{Orbit}(\Theta(V_{\check \CO})))\qquad \quad \quad \textrm{by Lemma \ref{theta3}}\\
&=&\nabla(\mathrm d_{\mathrm{BV}}(\check \CO')) \qquad  \quad \quad \qquad \,\,\textrm{by \eqref{orbitt}}\\
&=&\tilde{\mathrm d}_{\mathrm{BV}}(\check \nabla(\check \CO')) \qquad  \quad \quad \qquad \,\,\textrm{by Proposition \ref{cdnn}}\\
&=&\tilde{\mathrm d}_{\mathrm{BV}}(\check \CO).
\end{eqnarray*}
This completes the proof of Theorem \ref{thm16}.

\begin{remark} The above proof did not use the unitarity of  $\Theta(V_{\check \CO})$. Nonetheless, the unitarity of $V_{\check \CO}$ (and the stable range) was used in order to conclude the irreducibility of $\Theta(V_{\check \CO})$.
\end{remark}

\def\cCO{{\check{\CO}}}

\section{Weyl group representations and second proof of Theorems \ref{thm13} and \ref{thm16}}
\label{sec:proof2}

The purpose of this section is to give another proof of Theorems \ref{thm13} and \ref{thm16}, based on the Vogan duality (see \cite{VoIC4} and \cite[Chapter 21]{ABV}). (For connected reductive complex Lie groups, this amounts to certain relation between the Kazhdan-Lusztig polynomials in \cite{KL}.)

In this second proof, we will compute the associated variety of the annihilator ideal of a certain irreducible representation of $G=\Sp_{2n}(\C)$. When a representation has integral infinitesimal character, the answer is found in \cite{BVUni} and is conveniently expressed, via the Springer correspondence, in terms of Weyl group representations. The reduction step to integral infinitesimal character, again via the Springer correspondence, requires us to apply the $j$-operation (or the truncated induction) of Weyl group representations. The proof of Theorems \ref{thm13} and \ref{thm16} thus boils down to a certain abstract property as well as a concrete computation of the $j$-operation (from the integral Weyl group to the Weyl group).

Along the way of the proof, we also introduce the notion of metaplectic double cells and metaplectic special representations.

\subsection{Reduction to integral infinitesimal character}\label{secred}

In this subsection only, let $G$  be an arbitrary connected reductive complex Lie group.  Let $H$ be a Cartan subgroup of $G$ and let $\lambda: H\rightarrow \C^\times$ be a character. The complexified differential of $\lambda$ has the form
\[
  (\lambda_L, \lambda_R): \h\otimes_\R\C=\h\times \bar \h\rightarrow \C,
\]
where $\h$ denotes the Lie algebra of $H$, and $\bar \h$ is the complex Lie algebra equipped with a conjugate linear isomorphism $\h\rightarrow \bar \h,\ x\mapsto \bar x$. Write $\overline{\lambda_R}\in \h^*$ for the composition of
\[
 \h\xrightarrow{x\mapsto \bar x}\bar \h\xrightarrow{\lambda_R}\C\xrightarrow{\textrm{complex conjugation}}\C.
\]
Then  $\lambda_L-\overline{\lambda_R}\in \h^*$ is the differential of a holomorphic character of $H$. In particular, it is a  weight in the sense that
\be\label{intl}
\la \check \alpha, \lambda_L-\overline{\lambda_R}\ra \in \Z\quad \textrm{for every coroot $\check \alpha\in \h$ of $G$. }
\ee

Write $\widecheck G$ for the Langlands dual group of $G$. Let
$\widecheck H$  be a Cartan subgroup of it. As usual, we identify the root datum of $(\widecheck G, \widecheck H)$ with the dual of that of $(G,H)$.
Write $\widecheck G(\lambda)$ for the connected complex reductive subgroup of $\widecheck G$ containing $\widecheck H$ whose root system equals
\[
 \{\check \alpha\in \h \textrm{ is coroot of $G$}\,:\, \la \check \alpha, \lambda_L\ra \in \Z\}.
\]
Write $G(\lambda)$ for the Langlands dual of $\widecheck G(\lambda)$. As usual, we also view $H$ as a Cartan subgroup of $G(\lambda)$.

Write $\overline X_G(\lambda)$ for the irreducible subquotient of $X_G(\lambda)$ containing the  $K$-type of extremal weight $\lambda|_T$, where $T$ is the maximal compact torus in $H$, $K$ is a maximal compact subgroup of $G$ containing $T$, and $ X_G(\lambda)$ is the principal series representation of $G$ induced from the character $\lambda$ by normalized smooth parabolic induction. Here we fix an arbitrary  Borel subgroup of $G$ containing $H$ for the parabolic induction, and the isomorphism class of the representation $\overline X_G(\lambda)$ is  independent of the choices of $K$ and the Borel subgroup. By the work of Zhelobenko, every irreducible admissible representation of $G$ is of this form up to infinitesimal equivalence. Likewise we have the irreducible representation $\overline X_{G(\lambda)}(\lambda)$ of $G(\lambda)$. Note that $\overline X_{G(\lambda)}(\lambda)$ has integral infinitesimal character, by its very definition.

For a finite group $E$, let $\Irr(E)$ denote the set of isomorphism classes of irreducible complex representations of $E$.
Let $W\subset \GL(\h)$ denote the Weyl group of $G$. For every irreducible admissible
representation $\pi$ of $G$, write
\[\sigma(\pi ):=\mathrm{Springer} (\orb(\pi ))\in \Irr(W).\]
Here and henceforth, ``$\mathrm{Springer}$" indicates the irreducible representation of the Weyl group corresponding to the trivial local system on the nilpotent orbit under Springer correspondence \cite{Spr}. Also ``$\orb$'' indicates the nilpotent orbit as in \Cref{p:orbit}. Specifically, we have an irreducible representation $\sigma(\overline{X}_G(\lambda))\in \Irr(W)$. Likewise  let $W(\lambda)\subset W$ denote the integral Weyl group, namely the Weyl group of $G(\lambda)$, and we have an irreducible representation $\sigma(\overline X_{G(\lambda)}(\lambda))\in \Irr(W(\lambda))$.

\def\barX{\overline{X}}
\def\barfgg{\overline{\fgg}}
\def\barfhh{\overline{\fhh}}
\begin{prop}\label{redc}
  Let the notation be as above. Then as irreducible representations of $W$,
  \[
  \sigma( \overline X_G(\lambda))\cong j_{W(\lambda)}^W \sigma(\overline X_{G(\lambda)}(\lambda))).
  \]
  Here the symbol $j$ denotes the $j$-operation, as in \cite[Section 11.2]{Carter}.
\end{prop}
\begin{proof}
Fix a positive system of the root system of $(\fgg\times \barfgg, \fhh\times \barfhh)$ such that $(\lambda_L,\lambda_R)$ is dominant. Then there is a unique coherent family $\Psi_G$ such that $\Psi_G(\lambda) = \barX_G(\lambda)$
and $\Psi_G(\lambda')$ is irreducible when $\lambda'$ is regular dominant. See \cite[Chapter 7]{Vg}.
By \cite[Section 5]{J80} and \cite[Theorem~1.2]{King}, $\sigma( \overline X_G(\lambda))\otimes \sigma( \overline X_G(\lambda))$ is generated by the Bernstein degree polynomial $c_{\barX_G}(\lambda)$ of $\Psi_G$
as a $W\times W$-subrepresentation in $S(\fhh\times \barfhh)$.
A similar statement applies to $\barX_{G(\lambda)}(\lambda)$.

By the Vogan duality \cite[Theorem 1.15]{VoIC4}, $c_{\barX_G}(\lambda)$ and $c_{\barX_{G(\lambda)}}(\lambda)$ are equal up to a scalar multiple.
Note that $c_{\barX_{G(\lambda)}}(\lambda)$ is a $W(\lambda)\times W(\lambda)$-harmonic polynomial (invariant under the action of the diagonal $W(\lm)$) (see \cite{BVlocal}*{Proposition~4.7}).  The proposition follows by the definition of $j$-operation (\cite{Carter}*{p. 368}).
\end{proof}

\subsection{Metaplectic double cells and metaplectic special representations}
\label{sec:mcell}
Recall from \Cref{sec:intro} that $\sfW_{n}$ ($n\geq 0$) is the Weyl group of type $\mathrm C_n$, 
which is identified with the subgroup of $\GL_n(\bC)$ generated by the permutation matrices and the diagonal matrices with diagonal  entries $\pm 1$. It is viewed as a Coxeter group with a system of simple reflections determined by the positive root system (see \eqref{eq:roots}):
\[
  \Phi^+=\{e_i\pm e_j\mid 1\leq i<j\leq n\}\sqcup \{2e_i\mid 1\leq i\leq n\}\subset \C^n.
\]
Let $\sfW'_{n}$ be the Weyl
group of type $\mathrm D_{n}$, identified as a subgroup of $\sfW_{n}$ in the standard way, to be viewed as a Coxeter group with a system of simple reflections determined by the positive root system
\[
  \Psi^+:=\{e_i\pm e_j\mid 1\leq i<j\leq n\}\subset \C^n.
\]

We use $\sim$ to denote the (usual) double cell relation on $\Irr(\sfW'_{n})$, see \cite{Lu.CF}*{Chapter~4} and \cite{Carter}*{13.2}. Then the natural  action of $\sfW_n$ on $\Irr(\sfW'_{n})$ descends to an action on $\Irr(\sfW'_{n})/\sim $. We define big double cells in $\Irr(\sfW'_{n})$ by gluing double cells in the same orbit of the $\sfW_n$-action, as follows.
Given $\sigma'_{1}, \sigma'_{2} \in \Irr(\sfW'_{n})$, we define $\sigma'_{1}\osim \sigma'_{2}$ if there exists $s\in \sfW_{n}$ such that
$\sigma'_{1}\sim (\sigma'_{2})^{s}$.
The equivalent classes of $\osim$ are called big double cells in  $\Irr(\sfW'_{n})$.

We now define the notion of metaplectic double cells in $\Irr(\sfW_{n})$ in terms of
big double cells in $\Irr(\sfW'_{n})$, as follows. Given $\sigma_{1}, \sigma_{2}\in \Irr(\sfW_{n})$, we define $\sigma_{1}\msim \sigma_{2}$
if $\sigma'_{1}\osim \sigma'_{2}$ for some (hence all) irreducible constituents $\sigma'_{1}$ of
$\sigma_{1}|_{\sfW'_{n}}$ and $\sigma'_{2}$ of $\sigma_{2}|_{\sfW'_{n}}$.
The equivalent classes of $\msim$ are called metaplectic double cells in $\Irr(\sfW_{n})$.

By the Clifford theory and the explicit description of double cells in $\Irr(\sfW'_{n})$ (see \cite{Lu.CF}*{(4.5.6)} and \cite{BVPri2}*{Theorem~2.29}), we now describe metaplectic double cells. Let $\cC\in \Irr(\sfW_n)/\msim$.
We have two cases:
\begin{itemize}
    \item[\MA] If  there is
    $\sigma\in \cC$ such that $\sigma|_{\sfW'_n}$ is reducible,
     then  $\cC$ is the singleton $\set{\sigma}$. Moreover $\sigma|_{\sfW'_n}$ is a direct sum of two irreducible representations, say $\sigma'_1$ and $\sigma'_2$, and $\set{\sigma'_1}$ and $\set{\sigma'_2}$ are two distinct double cells in $\Irr(\sfW'_{n})$ forming an orbit under the $\sfW_n-$action.
     \item[\MB] Otherwise,
     there is a unique double cell $\cC'$ in $\Irr(\sfW'_n)$  such that
     \[
     \begin{split}
    \cC  & = \set{\sigma\in \Irr(\sfW_n) | \sigma|_{\sfW'_n}\in \cC'} \\
        & = \set{\sigma \in \Irr(\sfW_n) | \text{$\sigma$ occurs in $\Ind_{\sfW'_n}^{\sfW_n} \sigma'$, for some $\sigma'\in \cC'$
        }}.
    \end{split}
     \]
     Moreover, every element in $\cC'$ is fixed under the $\sfW_n$-action.
\end{itemize}
In summary, we have bijections
\[
(\Irr(\sfW_n)/\msim) \overset{\iotacell}{\longrightarrow}   (\Irr(\sfW'_n)/\osim )\overset{\iota_o } {\longrightarrow}\Irrsp(\sfW'_n)/\sfW_n, \]
where
\[
\begin{array}{rcl}
&\iotacell(\cC)&=\cC':=\set{\sigma'\in \Irr(\sfW'_n)| \text{$\sigma'$ occurs in some $\sigma\in \cC$}}, \\
&\iota_o (\cC')&:=\cC'\cap \Irrsp(\sfW'_n).
    \end{array}
  \]
 Here and henceforth,  a superscript ``$\mathrm{sp}$" over ``$\Irr$" indicates the subset of special representations.

    The following lemma should be compared with \cite{BVPri2}*{Corollary~2.16}. Recall the notion of the fake degree of  $\sigma \in \Irr(\sfW_{n})$ (see \cite[Definition 2.14]{BVPri2} where it is denoted by $a(\sigma)$).

\begin{lem}
\label{lem:defms}
  For each metaplectic double cell $\cC \in \Irr(\sfW_{n})/\msim$, there is a unique
  irreducible representation $\sigma_{\cC}$ in $\cC$ having the minimal fake degree.
\end{lem}
\begin{proof}
  The lemma is clear in the case \MA.
  Otherwise, we are in case \MB.
  Then $\cC':=\iotacell(\cC)$ is a double cell in $\Irr(\sfW'_n)$.
  Let $\sigma'_{\cC'}$ be the unique special representation in $\cC'$. Since the fake degree does not decrease under $\Ind_{\sfW'_n}^{\sfW_n}$,
  it is clear that $\sigma_{\cC}:=j_{\sfW'_n}^{\sfW_n} \sigma'_{\cC'}$ (see \cite{Carter}*{p. 368} for the definition of $j$-operation) satisfies the required property.
\end{proof}

We call the representation $\sigma_{\cC}$ in \Cref{lem:defms} a metaplectic special representation. Denote by $\Irrms(\sfW_n)$ the set of metaplectic special representations in $\Irr(\sfW_n)$.
Then we have a bijection
\[
    \begin{array}{rcl}
    (\Irr(\sfW_n)/\msim ) & \longrightarrow & \Irrms(\sfW_n)\\
    \cC & \mapsto & \text{the metaplectic special representation in $\cC$}.
    \end{array}
\]

We will use $\sgn$ to denote the sign character of a Weyl group (under consideration). We may occasionally put a subscript Weyl group under $\sgn$ in order to minimize possible confusion.

Recall the bijective map
$\iota_o\circ \iotacell$:
$(\Irr(\sfW_n)/\msim) \longrightarrow \Irrsp(\sfW'_n)/\sfW_n$. In fact we have the following

\begin{prop}\label{lemmpj}
The $j$-operation yields a bijective map
  \[
 j_{\sfW_n'}^{\sfW_n}: \Irr^{\mathrm{sp}}(\sfW'_n)/\sfW_n\rightarrow  \Irr^{\mathrm{ms}}(\sfW_n).
  \]
  Moreover, $ j_{\sfW_n'}^{\sfW_n}$ is compatible with twisting of the sign character:
  \begin{equation}
  \label{eq:ms.sgn}
j_{\sfW'_n}^{\sfW_n} (\sigma \otimes \sgn_{\sfW'_n}) =  (j_{\sfW'_n}^{\sfW_n}\sigma)\otimes \sgn_{\sfW_n}
\end{equation}
for each $\sigma \in \Irrsp(\sfW'_n)$.
\end{prop}
\begin{proof}
This is clear from the proof of \Cref{lem:defms}.
\end{proof}

\begin{prop}\label{prop:tdSP}
The diagram
\[
\begin{CD}
\overline{\mathrm{Nil}}^{\mathrm{sp}}(\o_{2n}(\C)) @>\tdSP>> \overline{\mathrm{Nil}}^{\mathrm{ms}}(\s\p_{2n}(\C)^*)\\
@V\mathrm{Springer}  VV                                     @VV\mathrm{Springer}   V\\
\Irr^{\mathrm{sp}}(\sfW_n')/\sfW_n @>j_{\sfW'_n}^{\sfW_n} >>  \Irr^{\mathrm{ms}}(\sfW_n)\\
\end{CD}
\]
commutes, where the map $\tdSP$ in the top horizontal arrow is defined  in \eqref{tdsp00}.
Here the left vertical arrow is the map induced by the Springer correspondence of $\fso_{2n}(\bC)$.
\end{prop}
\begin{proof}
  Take $\cO \in \bNilsp(\foo_{2n}(\bC))$.
  If the Young diagram of $\cO$ has only even columns,
  the proposition is easily verified.
  In general, $\cO$ can be written as
  \[
  \begin{split}
    \cO =&
    \Ind (\cO'\times \cO_0) \\
    & (\textrm{nilpotent $\rO_{2n}(\bC)-$orbit in $\s\o_{2n}(\C)$ induced from an appropriate Levi subalgebra}),\\
  \end{split}
  \]
  where $\cO'$ is a nilpotent orbit in a general linear Lie algebra whose Young diagram has only odd columns and $\cO_0$ is a nilpotent orbit in an orthogonal Lie algebra whose Young diagram has only even columns. See \cite{LS} or \cite[Section 7.1]{CM} for the notion of induced nilpotent orbit. Now one  verifies the following compatibility equation:
    \[
    \begin{split}
    \tdSP(\cO) =&  \Ind (\cO' \times \tdSP(\cO_0)) \smallskip \\
    & (\textrm{nilpotent orbit in $\s\p_{2n}(\C)$ induced from an appropriate Levi subalgebra}).
    \end{split}
    \]
Since Springer correspondence and $j$-operation are compatible with parabolic induction of nilpotent orbits (\cite[Theorem~3.5]{LS}), the proposition follows.
\end{proof}

\begin{cor}\label{mpsp}
  The  Springer correspondence yields a bijective map
  \[
  \mathrm{Springer}: \overline{\mathrm{Nil}}^{\mathrm{ms}}(\s\p_{2n}(\C)^*)\rightarrow \Irr^{\mathrm{ms}}(\sfW_n).
  \]
  Moreover, the inverse image of $\Irr^{\mathrm{ms}}(\sfW_n)\, (\subset \Irr(\sfW_n))$ under
  \begin{equation}\label{eq:springer}
   \mathrm{Springer}: \overline{\mathrm{Nil}}(\s\p_{2n}(\C)^*)\rightarrow \Irr(\sfW_n)
  \end{equation}
  equals $\overline{\mathrm{Nil}}^{\mathrm{ms}}(\s\p_{2n}(\C)^*)$.
\end{cor}
\begin{proof}
The first assertion follows from bijectivivity of the top, left and bottom arrows in \Cref{prop:tdSP}.
The second assertion follows from the injectivity of \eqref{eq:springer}.
\end{proof}

\subsection{Second proof of \Cref{thm13}}
We are back in the setting of Theorem \ref{thm13}. Recall that $\g=\s\p_{2n}(\bC)$, and $I$ is a primitive ideal of $\oU(\g)$ with a
metaplectic integral infinitesimal character $\chi_{I}$
(\Cref{def:mpli}), i.e. its coordinates are in
$(\frac{1}{2}+\BZ)^n$. Denote by $\CO\in
\overline{\mathrm{Nil}}(\g^*)$ the nilpotent orbit such that the
associated variety of $I$ equals the closure of $\CO$ (see \Cref{p:orbit}).

We let $G=\Sp_{2n}(\C)$, and  $V$ be as in Lemma \ref{exv}. Let $H$ be a Cartan subgroup of $G$. Take a character $\lambda:H\rightarrow \C^\times$ so that $V\cong \overline X_G(\lambda)$. Then $G(\lambda)=\SO_{2n}(\C)$, and $\orb(\overline X_{G(\lambda)}(\lambda))$ is a special nilpotent $\SO_{2n}(\C)$-orbit in $\s\o_{2n}(\C)^*$ (\cite[Definition 1.10]{BVUni} and the paragraph immediately after). Thus \[
  \sigma(\overline X_{G(\lambda)}(\lambda))\in \Irr^{\mathrm{sp}}(\sfW_n').
\]
It follows from \Cref{redc} and \Cref{lemmpj} that
\[
  \sigma(\overline X_{G}(\lambda))\in \Irr^{\mathrm{ms}}(\sfW_n).
\]
By our construction of $V$, we have $\CO=\orb(\overline X_{G}(\lambda))$.  \Cref{mpsp} now implies that $\CO$ is metaplectic special. This proves Theorem A.

\subsection{Second proof of \Cref{thm16}}
We are back in the setting of Theorem \ref{thm16}. Recall that $\g=\check \g=\s\p_{2n}(\bC)$,
$\check \CO\in \overline{\mathrm{Nil}}(\check \g)$ and $I_{\check \CO}$ is the maximal
ideal of $\oU(\g)$ with infinitesimal character $\chi_{\check \CO}$.

We retain the notation of Section \ref{secred}.
We again let $G=\Sp_{2n}(\C)$ so we have an identification  $H=(\C^\times)^n$ as usual. Then $\h=\bar \h=\C^n$. Suppose that $(\lambda_L, \lambda_R)=(\chi_{\check \CO}, \chi_{\check \CO})$.
We identify the nilpotent orbit $\ckcO$ with its Young diagram.
Let $\ckcO_g$ be the Young diagram formed by the even rows of $\ckcO$.
Since odd rows occur with even mulplicity in $\ckcO$, we let $\ckcO'_b $ be the Young diagram formed by odd rows of $\ckcO$ but with half multiplicity.
Then
\[
\widecheck G(\lambda)=\SO_{2b+1}(\C)\times \SO_{2g}(\C)\quad\textrm{ and }\quad G(\lambda)=\Sp_{2b}(\C)\times \SO_{2g}(\C),
\]
where $b$ is the size of $\ckcO'_b$ and $2g$ is the size of $\ckcO_g$.
The integral Weyl group is of the form $W(\lambda) = \sfW_b \times \sfW'_g$.

By \cite[Section~5]{BVUni}, we have
\[
  \sigma(\overline X_{G(\lambda)}(\lambda))=
  \left(j_{W_\lambda}^{W(\lambda)} \sgn\right)\otimes \sgn
  =: \sigma_b\otimes \sigma_g \in \Irr(\sfW_b)\times  \Irr(\sfW'_g),
\]
where $W_\lambda$ is the stabilizer of $\lambda_L$ in $W$.

Let $\sfS_k$ ($k\geq 0$) be the symmetric group in $k$ letters. We have (by a  routine computation)
\begin{itemize}
\item $\sigma_b  = j_{\prod_{i} \sfS_{b_i}}^{\sfW_b} \sgn  = j_{\sfS_b}^{\sfW_b} (\ckcO'_b)^{\mathrm t}$; and
\item  $\sigma_g = \left( j_{\prod_{i} \sfS_{c_{2i}}}^{\sfW'_g}\sgn \right)\otimes \sgn_{\sfW'_g}$
\end{itemize}
where $b_i$ is the length of $i$-th row in $\ckcO'_b$ and
$c_{2i}$ is the length of $(2i)$-th column in $\ckcO_g$.
(Note that the columns of $\ckcO_g$ occur in pairs.)  In addition we have identified the Young diagram $(\ckcO'_b)^{\mathrm t}$ with the irreducible representation of $\sfS_b$ via the Springer correspondence for $\GL_b(\bC)$ (see \cite[11.4 and 13.3]{Carter}):
\[
  \Springer( (\ckcO'_b)^{\mathrm t}) = j_{\prod_{i}\sfS_{b_i}}^{\sfS_{b}} \sgn.
\]
By \Cref{redc}, $j$-operation by stage \cite[Proposition 11.2.4]{Carter} and \eqref{eq:ms.sgn},
\[
\begin{split}
\sigma(\barX_G(\lambda)) &  = j_{\sfW_b\times \sfW'_g}^{\sfW_n}\sigma_b\otimes \sigma_g \\
 &  = j_{\sfS_b\times \sfW'_g}^{\sfW_n}(\ckcO'_b)^{\mathrm{t}}\otimes \left(
 \left( j_{\prod_{i} \sfS_{c_{2i}}}^{\sfW'_g}\sgn \right)\otimes \sgn_{\sfW'_g}\right)\\
 &  = j_{\sfS_b\times \sfW_g}^{\sfW_n}(\ckcO'_b)^{\mathrm{t}}\otimes \left(
 \left( j_{\prod_{i} \sfS_{c_{2i}}}^{\sfW_g}\sgn \right)\otimes \sgn_{\sfW_g}\right)\\
\end{split}
\]
By \cite{BM}, there is a (unique) nilpotent orbit $\cO_g\in \overline{\Nil}(\s\p_{2g}(\bC))$ such that
\[
\Springer(\cO_g) = \left( j_{\prod_{i} \sfS_{c_{2i}}}^{\sfW_g}\sgn \right)\otimes \sgn_{\sfW_g}.
\]
By direct computation (see \cite[Section~7]{Sommers} for example),  we have
\[
    \cO_g = \tdBV(\ckcO_g).
\]
Since Springer correspondence and $j$-operation are compatible with parabolic induction of nilpotent orbits, we have
\[
    \orb(\barX_G(\lambda)) = \Ind_{\gl_b(\C)\times \fsp_{2g}(\bC)}^{\fsp_{2n}(\bC)} (\ckcO'_b)^{\mathrm{t}}\times \cO_g.
\]
By the definition of $\tdBV$, we also have
\[
\tdBV(\ckcO) = \Ind_{\gl_b(\C)\times \fsp_{2g}(\bC)}^{\fsp_{2n}(\bC)} (\ckcO'_b)^{\mathrm{t}} \times \tdBV(\ckcO_g).
\]
Combining the above two equations, we have
\[
    \orb(\barX_G(\lambda)) =
    \tdBV(\ckcO)
\]
and this proves the theorem.

\subsection{Metaplectic Barbasch-Vogan duality in terms of metaplectic double cell}

The above proof of Theorem \ref{thm16} also implies the following result, which should be compared with \cite{BVUni}*{Proposition~A2}.
\begin{prop}\label{cd22}
  The diagram
  \[
\begin{CD}
\overline{\mathrm{Nil}}(\s\p_{2n}(\C)) @>\tdBV>> \overline{\mathrm{Nil}}^{\mathrm{ms}}(\s\p_{2n}(\C)^*)\\
@V\mathrm{Springer}  VV                                     @VV\mathrm{Springer}   V\\
\Irr(\sfW_n) @>>>  \Irr^{\mathrm{ms}}(\sfW_n)\\
\end{CD}
\]
commutes,
where the bottom horizontal arrow is the map that sends $\sigma\in \Irr(\sfW_n)$ to the  unique metaplectic special representation of $\sfW_n$  in the metaplectic double cell containing $\sigma\otimes \sgn$.
\end{prop}

\end{document}